\documentclass[11pt]{amsart}
\usepackage{etex,amsmath,amsthm,amssymb,amscd,amsfonts,setspace}
\usepackage{graphicx}    

\usepackage{verbatim}
\usepackage{pictex}
\usepackage{hyperref}

\newtheorem{theorem}{Theorem}[section]
\newtheorem{corollary}[theorem]{Corollary}
\newtheorem{lemma}[theorem]{Lemma}
\newtheorem{remark}[theorem]{Remark}
\newtheorem{proposition}[theorem]{Proposition}

\theoremstyle{definition}
\newtheorem{definition}[theorem]{Definition}

\newcommand{\CC}{\mathbb{C}}

\newcommand{\RR}{\mathbb{R}}
\newcommand{\ZZ}{\mathbb{Z}}

\newcommand{\Sharp}{\nolinebreak\hspace{-.05em}\raisebox{.6ex}{\scriptsize\bf \#}}

\begin{document}

\begin{center}
\end{center}

\title[Cayley graphs of reflection groups]
{Hamiltonian cycles in Cayley graphs of \\
imprimitive complex reflection groups}
\author{Cathy Kriloff and Terry Lay}
\address{Department of Mathematics \\
Idaho State University \\
Pocatello, ID 83209-8085}
\email{krilcath@isu.edu}
\email{layterr@gmail.com}

\begin{abstract}
Generalizing a result of Conway, Sloane, and Wilkes for real reflection groups, we show the Cayley graph of an imprimitive complex reflection group with respect to standard generating reflections has a Hamiltonian cycle.
This is consistent with the long-standing conjecture that for every finite group, $G$, and every set of generators, $S$, of $G$ the undirected Cayley graph of $G$ with respect to $S$ has a Hamiltonian cycle.
\end{abstract}
\maketitle

%


\section{Introduction}
\label{sec:intro}

For a finite group $G$ and a subset $S$ of $G\setminus\{1\}$, the (right, undirected) \textit{Cayley graph} of $G$ with respect to $S$, $\Gamma(G,S)$, has vertices corresponding to the elements $g \in G$ and edges $(g,gs)$ and $(g,gs^{-1})$ for each $g \in G$ and $s \in S$.
The Cayley graph is vertex-transitive, regular, and connected when $S$ generates $G$, which we assume throughout.  
Label the edge from $g$ to $gs$ by $s$ and that from $gs$ to $g$ by $s^{-1}$ (note that edge labels for $\Gamma(G,S)$ are drawn from $S\cup S^{-1}$).
It is common to consider both together as a single undirected edge with $s$ and $s^{-1}$ indicating travel along the edge in the appropriate direction. 

A \textit{path} in $\Gamma$ is an ordered sequence of adjacent vertices in $\Gamma$ and a path is \textit{self-avoiding} if no vertex appears more than once.
A \textit{Hamiltonian path} is a self-avoiding path containing every vertex of $\Gamma$.
When the initial and final vertex of a Hamiltonian path are adjacent it determines a \textit{Hamiltonian cycle} and a graph containing a Hamiltonian cycle is called \textit{Hamiltonian}.

The question, dating back to 1969 in a monograph by Lov\'{a}sz, of whether every connected vertex-transitive graph has a Hamiltonian path, remains unresolved.
The stronger claim, that every connected vertex-transitive graph has a Hamiltonian cycle, is known to be false and it has been observed that the four known counterexamples are not Cayley graphs.
The resulting conjecture that for every finite group $G$ and any generating set $S$ the Cayley graph $\Gamma(G,S)$ has a Hamiltonian cycle also remains unresolved and finding such a cycle is an NP-complete problem in general.
See~\cite{Gallian-Witte84,Curran-Gallian96,Kutnar-Marusic09,Pak-Radoicic09} for surveys of the status and history of the problem and references, including those supplying counter-conjectures.
When $S$ is not closed under inversion, it is possible for the directed graph with vertices the elements of $G$ and only edges $(g,gs)$ for $g \in G$ and $s \in S$ to have no Hamiltonian cycle.
For instance, the directed circulant graph on $\ZZ_{12}$ with generators $3, 4$ (and $6$) is not Hamiltonian (see~\cite{Gallian-Witte84} and~\cite{Locke-Witte99}). 

The conjecture that every (undirected) Cayley graph is Hamiltonian is easy to prove for abelian groups and known to be true for several specific types of groups that are nearly abelian, with either specific or arbitrary generating sets.  
For instance the conjecture has been shown true when
\begin{itemize}
\item $G$ is a $p$-group~\cite{Witte86}, 
\item the commutator subgroup $G'$ is a cyclic $p$-group~\cite{Marusic83, Durnberger83, Durnberger85, Keating-Witte85, DGMW98}, 
\item the order of $G$ has few prime factors~\cite{KMMMS12}, 
\item the order of $G$ is odd and $G'$ has order $pq$ or is cyclic of order $p^a q^b$ for $a,b\geq 0$~\cite{Witte-Morris12},
\item $G$ is nilpotent and $G'$ is cyclic~\cite{Ghaderpour-Morris14}.
\end{itemize}
It is also known that the Cayley graph of the semidirect product of two cyclic groups with respect to a specific generating set is Hamiltonian~\cite{Alspach89}.
We prove as our main result (Theorem~\ref{th:cximprimHam}) that the conjecture is true for the highly non-abelian infinite family of complex reflection groups, $G=G(de,e,n)\cong \mathbf{\mu}^n \rtimes S_n$
with respect to commonly used generating sets of reflections.  Here $\mathbf{\mu}$ is the cyclic group of $de$-th roots of unity.

\textbf{Main Result.}
\textit{If $G$ is an irreducible imprimitive complex reflection group and $S$ is a standard generating set for $G$, then the (undirected right) Cayley graph $\Gamma(G,S)$ has a Hamiltonian cycle.}

Our result generalizes that in~\cite{CSW89}, which provides an algorithm to generate a Hamiltonian cycle in each $\Gamma(G,S)$ where $G$ is a finite real reflection group and $S$ is the standard set of generating simple reflections.
That paper utilized the Coxeter presentation of the groups to give an inductive proof of the existence, and hence recursive construction, of a Hamiltonian cycle.
It also explicitly treats the small number of base cases.

Although there is no such uniformly well-behaved presentation or set of generators for complex reflection groups (see~\cite{Bremke-Malle98, Shi02}), we use those that go back to~\cite{Coxeter67, Cohen76} and are given in the standard references~\cite{BMR98, Lehrer-Taylor09}.  
While these presentations, generating sets, and resulting Cayley graphs do not tend to satisfy the usual conditions for the existence of Hamiltonian cycles currently given in the literature (see further discussion at the end of Section~\ref{sec:background}), they do allow for an inductive approach similar to that in~\cite{CSW89}.
In order to exploit that approach we must treat six infinite families of groups as base cases.
For three of these families we explicitly write down Hamiltonian cycles and in the remaining three cases our proofs provide a method for doing so.

In two cases we utilize a process we call flipping which is sometimes referred to in the literature as a P\'{o}sa exchange and is 
similar to a process utilized in a probabilistic algorithm to find Hamiltonian cycles in general graphs (see~\cite{Angluin-Valiant79}).
It would be interesting to further explore the application of the flipping process and related algorithms to Cayley graphs of complex reflection groups and in particular, to determine whether there exist obstructions to their success in these graphs and if so, under what conditions.  
Another interesting direction would be to investigate how our result might be of use in group coding (see~\cite{KNS14}).

The paper is organized as follows.
In Section~\ref{sec:background}, following~\cite{Geck-Malle06}, we review necessary facts about real reflection groups and summarize the classification of complex reflection groups due to Shephard and Todd~\cite{Shephard-Todd54}.  
The classification consists of a three-parameter infinite family, $G(de,e,n)$, along with 34 exceptional groups.
This paper treats only the $G(de,e,n)$, though we have conducted some initial investigations for the exceptional complex reflection groups.
Generating sets $S$ for the $G(de,e,n)$ are also given in Section~\ref{sec:background}.
Our formulations of the commonly used Factor Group Lemma (see~\cite{KMMMS12}), the flipping process, and the method of lifting cycles from quotient graphs are described in Section~\ref{sec:lift}.
The main result 
and its proof, including all base case lemmas, appear in Section~\ref{sec:cxgroupgraphs}.

\section{Background on reflection groups}
\label{sec:background}

A \textit{Coxeter system} $(G,S)$ is a group $G$ with set of generators $S=\{s_1,\dots,s_n\}$ that has a presentation of the form 
\[G=\langle s_1,\dots,s_n \mid (s_i s_j)^{m_{ij}} \rangle,\]
where $m_{ii}=1$ and $m_{ij}=m_{ji}\geq 2$ for $1\leq i\neq j \leq n$.
Such a group is called a \textit{Coxeter group} and is more commonly denoted $W$ due to the connection with Weyl groups.
The presentations of Coxeter groups are classified using diagrams that graphically encode the $s_i$ and the $m_{ij}$. 
The classification of finite irreducible Coxeter groups consists of four infinite families and six exceptional groups.
Finite Coxeter groups have a geometric incarnation as they are exactly the finite groups generated by orthogonal reflections of a real vector space (see~\cite{Bourbaki02,Humphreys90}).
It can also be shown that there is a natural set of generating reflections up to conjugacy, so such a choice is fixed and these are termed \textit{simple reflections}.
Every Cayley graph $\Gamma(G,S)$ with $G$ a finite irreducible Coxeter group and $S$ a set of simple reflections is shown in~\cite{CSW89} to have a Hamiltonian cycle.

The notions of reflection and classification of finite groups generated by reflections extend to the setting of an $n$-dimensional complex vector space $V$ (for a brief survey, see~\cite{Geck-Malle06}).
A linear transformation $r:V \to V$ is a \textit{reflection} if it is of finite order and has a $+1$-eigenspace of dimension $n-1$.
In the remaining complex dimension the reflection acts by a root of unity and hence may have order greater than two.
A finite subgroup, $G$, of $\mathrm{GL}(V)$ generated by reflections is called a \textit{reflection group on $V$}.
Since $G$ is finite, the standard averaging technique makes it possible to fix a non-degenerate $G$-invariant hermitian form on $V$ and consider $G$ as a subgroup of the unitary group on $V$.
Finiteness of $G$ also guarantees the representation on $V$ is completely reducible, which means it suffices to consider reflection groups and spaces on which they act irreducibly.
More precisely, $G$ is said to act irreducibly in dimension $k$ if its fixed point space is of dimension $n-k$ and it acts irreducibly when restricted to the complement of that fixed point space. 

We describe an infinite family of complex reflection groups.
Let $d,e,n \geq 1$, let $\mathbf{\mu}$ denote the cyclic group of $de$-th roots of unity, and let $\zeta$ generate $\mathbf{\mu}$.
Under the \textit{standard monomial representation}, $G(de,e,n)$ consists of 
\begin{align*}
 &\hbox{monomial matrices with nonzero entries $\zeta^{a_1},\dots, \zeta^{a_n}$}, \\
 &\hbox{such that $(\zeta^{a_1}\cdots \zeta^{a_n})^d=1$, or equivalently $a_1+\cdots +a_n \equiv 0 \bmod{e}$.}
\end{align*}
Each such monomial matrix may be written as a product of a diagonal matrix with entries $\zeta^{a_1},\dots,\zeta^{a_n}$ and a permutation matrix (obtained by permuting columns of the identity matrix).
This provides an alternative description of $G(de,e,n)$ as an index $e$ subgroup of $\mathbf{\mu}\wr S_n=\mathbf{\mu}^n \rtimes S_n$, and makes it clear that $|G(de,e,n)|=d^n e^{n-1} n!$.
In this perspective, $G(de,e,n)$ consists of all
\[(a_1,\dots,a_n \mid \sigma) \hbox{ such that } a_i \in \ZZ_{de},\, a_1+\cdots +a_n \equiv 0 \bmod{e}, \hbox{ and } \sigma \in S_n,\]
and the action of $S_n$ on $\mathbf{\mu}^n$ providing the semidirect product structure on $G(de,e,n)$ is
\[\sigma.(a_1,\dots,a_n)=(a_{\sigma(1)},\dots,a_{\sigma(n)}),\]
so that if $\sigma, \tau \in S_n$,
\[(a_1,\dots,a_n \mid \sigma)(b_1,\dots,b_n \mid \tau)=(a_1+b_{\sigma(1)},\dots,a_n+b_{\sigma(n)} \mid \sigma\tau),\]
where $\sigma\tau$ is computed by applying $\sigma$ first, $\tau$ second.
This is consistent with the use of the right Cayley graph and with matrix multiplication where $(a_1, a_2, \dots , a_n \mid \sigma)$ represents the $n \times n$ matrix whose only nonzero entries are the $\zeta^{a_i}$ in position $i,\sigma(i)$ for $1\leq i\leq n$. 

A reflection group $G$ on $V$ is \textit{imprimitive} if there is a decomposition $V=V_1\oplus V_2 \oplus \cdots \oplus V_k$ into proper nonzero subspaces such that $G$ permutes the subspaces.
The groups $G(de,e,n)$ with $de,n\geq 2$ are imprimitive in their action on a system of lines orthogonal to the reflecting hyperplanes and the exceptional groups are primitive.
The symmetric groups $G(1,1,n)$ do not act irreducibly in the standard monomial representation but do act irreducibly on the complement of the span of the sum of all the basis vectors and are primitive on this $(n-1)$-dimensional subspace. 

The irreducible finite reflection groups were classified in~\cite{Shephard-Todd54} (see also~\cite{Lehrer-Taylor09}) and consist of the
\begin{itemize}
\item groups $G(de,e,n)$, with $de\geq 2$, $n\geq 1$, and $(de,e,n)\neq (2,2,2)$, which are imprimitive and irreducible in dimension $n$, 
\item symmetric groups $G(1,1,n)$, which are primitive and irreducible in dimension $n-1$, and
\item $34$ primitive exceptional groups, numbered $G_4,\dots,G_{37}$, irreducible in dimensions $2$ through $8$.
\end{itemize}

The reason for the numbering is that the original classification listed $G(d,1,1)$, $G(1,1,n)$ and $G(de,e,n)$ with $de,n \geq 2$ separately.
If there is a $G$-invariant real subspace, $V_0$, of $V$ so that the canonical map $\CC \otimes_{\RR} V_0 \to V$ is a bijection, then $G$ is a \textit{real reflection group}.
The finite real reflection groups occur in the classification from~\cite{Shephard-Todd54} as:
\begin{align*}
G(1,1,n) &\hbox{ of type } A_{n-1}, \hbox{ the symmetric group } S_n, \\
G(2,1,n) &\hbox{ of type } B_n, \hbox{ the binary octahedral group } \{\pm 1\}^n \rtimes S_n, \\
G(2,2,n) &\hbox{ of type } D_n, \hbox{ an index two subgroup of } \{\pm 1\}^n \rtimes S_n, \\
G(m,m,2) &\hbox{ of type } I_2(m), \hbox{ the dihedral group of order } 2n, \hbox{ and} \\
G_{23}, G_{28}, G_{30}&, G_{35}, G_{36}, \hbox{ and } G_{37} \hbox{ of types } H_3, F_4, H_4, E_6, E_7, \hbox{ and } E_8 \hbox{ respectively}.
\end{align*}

The following explicit choices of representations of generators for the groups are provided in~\cite{Lehrer-Taylor09} and are consistent with the presentations provided by the diagrams in~\cite{BMR98}. 
For $1\leq i <n$, let 
\[r_i=(0,\dots,0 \mid (i\,\,i+1))=\hbox{the identity matrix with columns $i$ and $i+1$ interchanged,}\]
\begin{align*}
s&=(-1,1,0,\dots,0 \mid (1\,\,2))
=\begin{bmatrix}
0 & \zeta^{-1} & 0 & \dots & 0 \\
\zeta & 0 & 0 & \dots & 0 \\
0 & 0 & 1 & \dots & 0 \\
0 & 0 & \dots & \ddots & 0 \\
0 & \dots & \dots & 0 & 1
\end{bmatrix}\!, \hbox{ and } \\
t&=(e,0,\dots,0 \mid 1)
=\begin{bmatrix}
\zeta^e & 0 & \dots & 0 \\
0 & 1 & \dots  & 0 \\
0 & 0 & \ddots & 0 \\
0 & \dots & 0 & 1
\end{bmatrix}.
\end{align*}
Then, using $e=1$ in $t$ for $G(d,1,n)$,
\begin{align*}
G(d,1,n)&=\langle t,r_1,r_2,\dots, r_{n-1}\rangle \hbox{ for } d\geq 2,n\geq 1, \\
G(e,e,n)&=\langle s,r_1,r_2,\dots, r_{n-1}\rangle,  \hbox{ for } e\geq 2, n\geq 1, \hbox{ and } \\
G(de,e,n)&=\langle s,t,r_1,r_2,\dots, r_{n-1}\rangle, \hbox{ for } d,e,n \geq 2.
\end{align*}
In particular, $G(d,1,n)$ and $G(e,e,n)$ are generated by $n$ reflections (are \textit{well-generated}) while the groups $G(de,e,n)$ with $d,e,n\geq 2$ require $n+1$ generators.

We note that the groups $G(de,e,n)$ do not tend to satisfy the sufficient conditions for the existence of Hamiltonian cycles currently known in the literature.
For example, Theorem~1.2 of~\cite{KMMMS12} and related references guarantee the existence of a Hamiltonian cycle in any Cayley graph of a finite group whose order has a ``small'' prime factorization of certain forms, and those of~\cite{Witte86, Witte-Morris12} apply to groups of prime power and odd order respectively.  
But such results will not apply in general to the $G(de,e,n)$. 
Since $G(de,e,n)$ is a semidirect product involving $S_n$, rather than cyclic groups of prime order and abelian groups, it is not of the form addressed in~\cite{Durnberger83} or~\cite{Alspach89}, and it is also not nilpotent for $n\geq 3$, so the results of~\cite{Ghaderpour-Morris14} on nilpotent groups do not generally apply.  
The results of~\cite{DGMW98, Keating-Witte85} apply to groups with commutator subgroups that are cyclic of prime power order.  
Using the presentations and the number of one-dimensional characters of $S_n$ and $\mu$, it is easy to count one-dimensional characters of $G=G(de,e,n)$ and the abelian group $G/G'$ and thus to compute $|G'|$.
This yields $|G'|=d^{n-1}$ when $e=n=1$, $|G'|=(de)^{n-1}n!/4$ when $n=2$ and $e$ is even, and $|G'|=(de)^{n-1}n!/2$ in the remaining cases.
This makes it clear that the results of~\cite{DGMW98, Keating-Witte85} do not generally apply.

Some Hamiltonicity results address groups with simple types of generators and relations.
For instance,~\cite{Glover-Marusic07} addresses groups generated by an involution and an element of order at least three whose product is of order three, while the three lemmas in Section 2 of~\cite{Pak-Radoicic09} involve generators that are involutions and/or satisfy very simple relations.
The generating sets we use typically do not meet these hypotheses on the number of, orders of, or relations between the generators. 
The main result in~\cite{Pak-Radoicic09} does provide, for an arbitrary group, the existence of a Hamiltonian cycle in its Cayley graph with respect to a relatively small generating set.  
While striking, this result serves a different purpose than we address here.
Lastly, several results in the broader graph theory literature on existence of Hamiltonian cycles treat graphs with sufficiently large degree and/or sufficiently small connectivity (see e.g.,~\cite{Dirac52, Jackson80}), which do not apply to the families of Cayley graphs we consider.

\section{Lifting cycles from quotient graphs}
\label{sec:lift}

Rather than specify a sequence of vertices, an equivalent way to describe a Hamiltonian cycle is to specify an initial vertex $v_1$ and a sequence of edges.
In the case of a Cayley graph it suffices to use $v_1=1$ and to list the sequence of edges as a sequence of generators, which we denote below with square brackets to connote an ordered list.
If $P=[s_1,s_2,\dots,s_{k-1},s_k]$ denotes the sequence of edges for a path, then $P\Sharp$ will denote $[s_1,s_2,\dots,s_{k-1}]$.

Our treatment of the base cases for the inductive proof of the main result for $G(de,e,n)$ involves the construction of explicit Hamiltonian cycles. 
We make use of the following techniques. 

\begin{lemma} 
\label{lem:l2.6}
Let $G$ be a group with generating set $S$ and let $\Gamma = \Gamma(G, S)$ denote its Cayley graph. 
  Let $B = [s_1, s_2, \dots, s_k]$, where each $s_i \in S$. Let  $v = s_1s_2 \cdots s_k$. Let $w_0 = 1$ and 
  for $0 < i < k$, let $w_i = s_1s_2 \cdots s_i$.  If $n > 0$,
  the path $P$ defined by $B^n \Sharp$ is
  self-avoiding if and only if whenever $0 \leq m < n$ and $v^m w_i = w_j$, it follows that  $m = 0$ and $i = j$.
\end{lemma}
\begin{proof}
 The forward direction is immediate. To establish the converse, observe that vertices 
 along the path $P$ are given by expressions of the form $v^m w_i$, where $0 \leq m < n$ and $0 \leq i < k$. Suppose that for 
 two such vertices we have 
 $v^{m_1} w_i = v^{m_2} w_j$. We can assume that $m_1 \leq m_2$. Cancellation yields 
 $v^m w_i = w_j$, where $m = m_2 - m_1 < n$. It follows that $m_1 = m_2$ and $i = j$, establishing that $P$ is self-avoiding. 
\end{proof}
 
\begin{remark}
\label{rem:factorgroup}
  Note that, in the context of Lemma~\ref{lem:l2.6},  
  $B^n \Sharp$ is a Hamiltonian path if it is self-avoiding and the size of the group $G$ is $nk$. The path $B^n$ is closed if $v$ is of order $n$. 
  When $B^n$ is a Hamiltonian cycle, the condition 
  \[v^m w_i = w_j \hbox{ implies } m = 0 \text{ and } i = j\] 
  is equivalent to 
  \[w_0, w_1, \cdots , w_{k-1} \text{ visits each left coset of } H = \langle v \rangle \text{ exactly once. }\] 
  See Lemma 2.6 in~\cite{KMMMS12}.  
\end{remark}
  
If a Hamiltonian path is known it may be possible to alter the path into a cycle with a process we call {\it flipping}.
See Figure~\ref{fig:flip1}.

\begin{center}
\begin{figure} [h!]
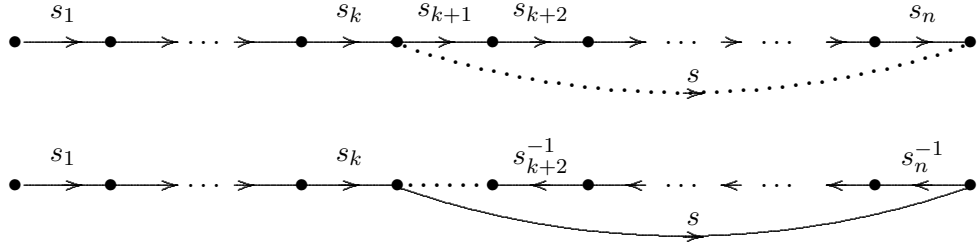

\begin{center}
\mbox
{
{
\beginpicture
     \setcoordinatesystem units <0.5in,0.5in>
      
     \setlinear
     \put {$\bullet$} at -7 1.5
     \put {$\bullet$} at -6 1.5
     \plot -6.9 1.5 -5.5 1.5 / 
     \plot -4.5 1.5 -.5 1.5 / 

     \plot 1.5 1.5 3 1.5 /

     \put {$\dots$} at -5 1.5

     \put {$\dots$} at 1 1.5
     \put {$\dots$} at 0 1.5
   
     \put {$\bullet$} at -4 1.5
     \put {$\bullet$} at -3 1.5
     \put {$\bullet$} at -2 1.5
     \put {$\bullet$} at -1 1.5
     \put {$\bullet$} at 2 1.5
     \put {$\bullet$} at 3 1.5
     
     \put{$s_1$} at -6.5 1.8
     \put{$s_k$} at -3.5 1.8
     \put{$s_{k+1}$} at -2.5 1.8
     \put{$s_{k+2}$} at -1.5 1.8
     \put{$s_n$} at 2.5 1.8
     
     \put{$s$} at .1 1.15
          
     \put {$\bullet$} at -7 0
     \put {$\bullet$} at -6 0
     \plot -6.9 0 -5.5 0 / 
     \plot -4.5 0 -3 0 / 
     \plot -2 0 -.5 0 / 

     \plot 1.5 0 3 0 /

     \put {$\dots$} at -5 0

     \put {$\dots$} at 1 0
     \put {$\dots$} at 0 0

     \put {$\bullet$} at -4 0
     \put {$\bullet$} at -3 0
     \put {$\bullet$} at -2 0
     \put {$\bullet$} at -1 0
     \put {$\bullet$} at 2 0
     \put {$\bullet$} at 3 0
 
     \arrow <6pt> [.2,.67] from -2.6 1.5 to -2.4 1.5 
     \arrow <6pt> [.2,.67] from -1.6 1.5 to -1.4 1.5 
     \arrow <6pt> [.2,.67] from -3.6 1.5 to -3.4 1.5 
     \arrow <6pt> [.2,.67] from -5.5 1.5 to -5.3 1.5 
     \arrow <6pt> [.2,.67] from -6.5 1.5 to -6.3 1.5 
     \arrow <6pt> [.2,.67] from 2.4 1.5 to 2.6 1.5 
     \arrow <6pt> [.2,.67] from .4 1.5 to .6 1.5 
     
     \arrow <6pt> [.2,.67] from -.6 1.5 to -.4 1.5 
     \arrow <6pt> [.2,.67] from -.4 0 to -.6 0 

     \arrow <6pt> [.2,.67] from 1.5 1.5 to 1.7 1.5 
     \arrow <6pt> [.2,.67] from 1.7 0 to 1.5 0 

     \arrow <6pt> [.2,.67] from -4.7 1.5 to -4.5 1.5 
     \arrow <6pt> [.2,.67] from -4.7 0 to -4.5 0

     \arrow <6pt> [.2,.67] from 0 .96 to .2 .96 

     \arrow <6pt> [.2,.67] from -1.4 0 to -1.6 0 
     \arrow <6pt> [.2,.67] from -3.6 0 to -3.4 0 
     \arrow <6pt> [.2,.67] from -5.5 0 to -5.3 0 
     \arrow <6pt> [.2,.67] from -6.5 0 to -6.3 0 
     \arrow <6pt> [.2,.67] from 2.6 0 to 2.4 0 
     \arrow <6pt> [.2,.67] from .6 0 to .4 0 
    
     \arrow <6pt> [.2,.67] from 0 -.54 to .2 -.54 
     
     \put{$s_1$} at -6.5 .3
     \put{$s_k$} at -3.5 .3
     \put{$s_{k+2}^{-1}$} at -1.5 .3
     \put{$s_n^{-1}$} at 2.5 .3
     
     \put{$s$} at .1 -.35
     
     \setsolid
     \circulararc 39 degrees from -3.05 0 center at 0 8.5

     \setdots
     \setplotsymbol({\Large .})
	 \plot -3 0 -2 0 /     
     
     \circulararc 39 degrees from -3.05 1.5 center at 0 10
     \setplotsymbol({\tiny -})
     \setsolid   
 
     \put{$\vspace{.5in}$} at -6 -.5
     
     \endpicture
}
}
\caption {\small Flipping a path}
 \label{fig:flip1}
\end{center}
\end{figure} 
\end{center}

\begin{definition}
If $P = [s_1,s_2, \dots , s_n]$ is a path and for some $k < n$ the vertices $s_1s_2\cdots s_k$ and $s_1s_2\cdots s_n$ are adjacent via an edge labeled $s$,
 then the {\it flip} of $P$ with respect to $s$, $F(P,s)$,  is the path $[s_1,s_2,\dots , s_k, s, s_n^{-1}, s_{n-1}^{-1}, \dots, s_{k+2}^{-1}]$. 
\end{definition}

The walks $P$ and $F(P,s)$ visit exactly the same vertices. 
In particular, if $P$ is self-avoiding, then $F(P,s)$ is self-avoiding.

Several of the proofs presented below make use of a lifting technique similar to that presented in the proof of Theorem~1.1 in~\cite{CSW89}. 
It relies on a method for combining disjoint cycles in the Cayley graph. 

\begin{definition}
\label{def:cjp}
If $C_1$ and $C_2$ are disjoint cycles in $\Gamma(G,S )$ and $r \in S$, we say that $C_1$ and $C_2$ satisfy the {\it commutative joining property} with respect to $r$ if there exist edges $(g_i,g_i s)$ in the $C_i$ and such that $g_2 = g_1 r$ and $r$ commutes with $s$.
\end{definition}

\begin{figure} [h!]
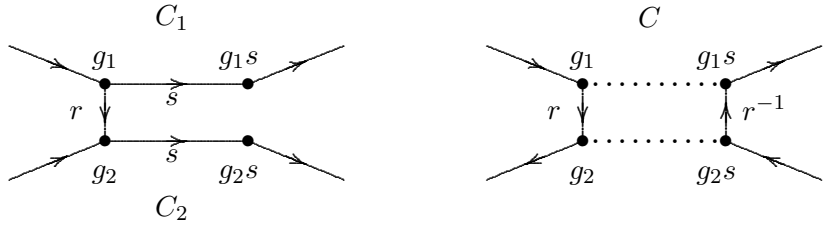

\begin{center}
\mbox
{
{
\beginpicture
     \setcoordinatesystem units <0.5in,0.5in>
     
     \setlinear
     
     \plot 0 1.4 1 1 /
     \plot 1 1 1 .4 /
     \plot 0 0 1 .4 /

     \arrow <6pt> [.2,.67] from 1.65 1 to 1.85 1     
     \arrow <6pt> [.2,.67] from 1.65 .4 to 1.85 .4     

     \arrow <6pt> [.2,.67] from .4 .16 to .6 .24      
     \arrow <6pt> [.2,.67] from 2.9 1.16 to 3.1 1.24     
     \arrow <6pt> [.2,.67] from 2.9 .24 to 3.1 .16     
     \arrow <6pt> [.2,.67] from .4 1.24 to .6 1.16
     \arrow <6pt> [.2,.67] from 1 .8 to 1 .6
     
     \plot 1 1 2.5 1 /
     
     \plot 1 .4 2.5 .4 /
     
     \plot 2.5 1 3.5 1.4 /
     
     \plot 2.5 .4 3.5 0 /
     
     \put {$s$} at 1.7 .85
     \put {$r$} at .7 .7
     \put {$s$} at 1.7 .25
 
     \put {$C_1$} at 1.7 1.7
     
     \put {$g_1$} at 1 1.25
     \put {$g_1s$} at 2.4 1.25

     \put {$g_2$} at 1 .05
     \put {$g_2s$} at 2.4 .05
     
     \put {$\bullet$} at 1 .4
     \put {$\bullet$} at 2.5 .4
 
     \put {$\bullet$} at 1 1
     \put {$\bullet$} at 2.5 1
     
     \put {$C_2$} at 1.7 -.3  
     
     \plot 5 1.4 6 1 /
     \plot 6 1 6 .4 /
     \plot 5 0 6 .4 /    
     
     \plot 7.5 1 8.5 1.4 /
     \plot 7.5 1 7.5 .4 /
     \plot 7.5 .4 8.5 0 /
     
     \put {$r$} at 5.7 .7
     \put {$r^{-1}$} at 7.9 .75
     
     \put {$\bullet$} at 6 .4
     \put {$\bullet$} at  7.5 .4
     \put {$\bullet$} at 6 1
     \put {$\bullet$} at  7.5 1
     
     \arrow <6pt> [.2,.67] from 5.6 .24 to 5.4 .16     
     \arrow <6pt> [.2,.67] from 7.9 1.16 to 8.1 1.24     
     \arrow <6pt> [.2,.67] from 8.1 .16 to 7.9 .24    
     \arrow <6pt> [.2,.67] from 5.4 1.24 to 5.6 1.16
     \arrow <6pt> [.2,.67] from 6 .8 to 6 .6
     \arrow <6pt> [.2,.67] from  7.5 .6 to 7.5 .8
     
     \setdots
     \setplotsymbol({\Large .})

     \plot 6 .4 7.5 .4 /
     \plot 6 1 7.5 1 /
     
     \setplotsymbol({\tiny -})
     \setsolid
     
     \put {$C$} at 6.7 1.7

     \put {$g_1$} at 6 1.25
     \put {$g_1s$} at 7.4 1.25

     \put {$g_2$} at 6 .05
     \put {$g_2s$} at 7.4 .05
     
\endpicture
}
}
\caption {\small Combining cycles with the commutative joining property}
\label{fig:commjoin}
\end{center}
\end{figure} 

When two cycles satisfy the commutative joining property they can be combined into a new cycle $C$ which spans the union of the vertex sets of $C_1$ and $C_2$. 
This is accomplished by first taking the edge $(g_1, g_1 r)$ from $C_1$ to $C_2$, traversing $C_2$ in reverse until arriving at $g_2 s$, then taking the edge $(g_2 s, g_2 s r^{-1})$. 
Because $r$ and $s$ commute, $g_2 s r^{-1} = g_1 r s r^{-1} = g_1 s$ and the cycle is completed by continuing on around $C_1$ to finish back at $g_1$.  
See Figure~\ref{fig:commjoin}.

The following propositions extend this joining strategy to sets of disjoint cycles.
\begin{proposition}
\label{prop:multi-join}
Let $G$ be a finite group with generating set $S$. Let $\{C,C_1,\dots,C_n\}$ be a collection of disjoint cycles in $\Gamma(G,S)$ and $r \in S$.
If $C$ and each $C_l$ have the commutative joining property with respect to $r$ then all of the cycles can be combined to form
a single cycle spanning the union of the various vertex sets.
 
\end{proposition}
\begin{proof}

For each cycle $C_l$, find distinct edge pairs $(g_l, g_ls_l)$ in $C$ and $(g_lr,g_ls_lr)$ in $C_l$ where $r$ and $s_l$ commute. 
When $C$ is combined with one of the other cycles as described above, the resulting cycle contains each of the edges $(g_l,g_ls_l)$ corresponding to the remaining $C_l$. 
Therefore, the process can be iterated until the cycle contains all of the $C_l$.
\end{proof}

\begin{proposition}
\label{prop:lifting}
Let $G$ be a finite group with generating set $S=\{r_1,\dots,r_n\}$, where $r_n$ is of order two. Let $H=\langle r_1,\dots,r_{n-1}\rangle$.  
If $\Gamma(H,S \setminus \{r_n\})$ has a Hamiltonian cycle $C_H$ with the property
\begin{equation*}
 \text{whenever $(g,gs)$, $(gs, gst)$ are consecutive edges in $C_H$, either $s$ or $t$ must commute with $r_n$},
 \end{equation*}
then $\Gamma(G,S)$ has a Hamiltonian cycle.
\end{proposition}
\begin{proof}

 Let $\{H_l \mid l = 1, \dots, k\}$ be the set of left cosets of $H$ with $H_1 = H$. Let $\Gamma_l$ denote the subgraph of $\Gamma(G,S \setminus \{r_n\})$ 
 with vertex set $H_l$. Each $\Gamma_l$ is isomorphic to $\Gamma(H,S \setminus \{r_n\})$. 
 Since $\Gamma(G,S)$ is connected, the set of subgraphs $\Gamma_l$ can be linked using edges labelled $r_n$.
 To start, set $C = C_H$.  
 
 Inductive hypothesis: $\Gamma(G,S)$ contains a cycle $C$ such that
 \begin{itemize}
 \item the vertex set of $C$ is the union of some subcollection of the left cosets of $H$, and
 
 \item if $r_i$ and $r_j$ are consecutive edge labels in $C$ and neither is 
 $r_n$ then $r_i$ and $r_j$ also appear as consecutive edge labels in the cycle $C_H$. 
 
 \end{itemize}
 
 The inductive hypothesis is trivially satisfied when $C=C_H$.
 If the vertex set of $C$ is all of $G$ we are finished. If not, there is an index $l$ such that
 $C$ and $\Gamma_l$ are disjoint and linked by $r_n$. Find consecutive 
 edges $(g^\prime, g^\prime r_i)$ and $(g^\prime r_i, g^\prime r_i r_j) = (g, g r_j)$ in $C$ and a vertex $g_l \in \Gamma_l$ 
 with $g r_n = g_l$. Observe that neither $r_i$ nor $r_j$ is $r_n$ since $r_n$ is of order 2.  
 
  From the inductive hypothesis the edge labels $r_i$ and $r_j$ must appear as consecutive edge labels in $C_H$. It follows that
  either $r_i$ or $r_j$ commutes with $r_n$. Assume that $r_n$ and $r_j$ commute; the case where $r_n$ commutes with $r_i$ is handled in similar fashion.
  
  The cycle $C_H$ must contain an edge of the form $(h, h r_j)$ for some $h \in H$. It follows that 
  the cycle $C_l = g_lh^{-1}C_H$ is a spanning cycle for $\Gamma_l$ and
  contains the edge $(g_l, g_l r_j)$. Since $C_l$ contains $(g_l,g_lr_j)$ and $gr_n=g_l$, we conclude that  $C$ and $C_l$ have the commutative joining property with respect to $r_n$ and 
  can be combined as described above, creating a 
 new cycle containing all of the vertices from $C$ and $C_l$, which we again call $C$.
 
 The new vertex set $C$ is a union of cosets of $H$. 
 Consecutive edge labels in $C$, neither of which is $r_n$, signify
 consecutive edges coming from either $C_l$ or the previous version of $C$ and, 
 in either case,  must appear as consecutive edge labels in $C_H$. The inductive hypothesis remains valid.
 
 Continuing in this manner we arrive at a cycle whose vertex set is the union of all of the left cosets of $H$ and is, therefore, a Hamiltonian cycle in $\Gamma(G,S)$.
\end{proof}

In the preceding proof, any choice of an edge labeled $r_n$ connecting $C$ and $\Gamma_l$ will work in the inductive step. 
In general, there will be many such edges.
If the conditions on $C_H$ are relaxed to allow the cycle to contain instances of consecutive edge labels signifying generators, neither of which commute with $r_n$, then the inductive step fails unless the edge connecting $C$ and $\Gamma_l$ can be chosen to avoid these ``bad'' portions of the copies of $C_H$ that have accumulated in $C$.

\begin{definition}
\label{def:badness}
Let $G$  be a finite group with generating set $S$. Let $C$ be a cycle in $\Gamma(G, S)$. For $r$ in $S$ define the {\it badness} of $C$ with respect to $r$, denoted
$bad(C,r)$, to be the number of instances in $C$ of consecutive edge labels $s$ and $t$ where $r$ fails to commute with both of the generators $s$ and $t$.
\end{definition}

\begin{corollary}
\label{cor:badness}

Let $G$ be a finite group with generating set $S=\{r_1,\dots,r_n\}$, where $r_n$ is of order two. Let $H=\langle r_1,\dots,r_{n-1}\rangle$.  
Let $\{\Gamma_l\}$ denote the components of $\Gamma(G, S \setminus \{r_n\})$. 
Denote by $c(i,j)$ the number of distinct edges in $\Gamma(G,S)$ that are  labeled $r_n$ and connect $\Gamma_i$ and $\Gamma_j$. 
If $\Gamma(H, S \setminus \{r_n\})$ contains a spanning cycle $C_H$ with $bad(C_H, r_n) < c(i,j)$ for all positive $c(i,j)$ then
  $\Gamma(G,S)$ contains a Hamiltonian cycle.
    
 \end{corollary}
\begin{proof}
  The proof proceeds exactly as the proof of Proposition \ref{prop:lifting}.   Each constructed $C_l$ 
  has the same sequence of edge labels as $C_H$ so that $bad(C_l,r_n) = bad(C_H,r_n)$.
  In the inductive step, 
  when a $\Gamma_l$ is chosen that is connected to $C$ by an edge labeled $r_n$, there must be $c(i,l)$ different vertices in $C$, 
  coming from some previously
  incorporated $C_i \subset \Gamma_i$,
  that connect $C$ to $\Gamma_l$ by applying $r_n$. The edges in $C$ that pass through these vertices come from $C_i$ and bear consecutive
  edge labels consistent with the sequence of edge labels in $C_H$. Since $c(i,l)$ exceeds $bad(C_H,r_n)$, a suitable choice of connecting edge can be made.
\end{proof}

\section{Application to Cayley graphs of complex reflection groups}
\label{sec:cxgroupgraphs}

Examining the diagrams and corresponding presentations in~\cite{BMR98} or the matrix forms of the generators given in Section~\ref{sec:background} shows that in
\begin{align*}
G(d,1,n) & \hbox{ with } d \geq 2, n \geq 3, \\
G(e,e,n) & \hbox{ with } e \geq 2, n \geq 4, \hbox{ and} \\
G(de,e,n)& \hbox{ with } d,e \geq 2, n \geq 4,
\end{align*}
$r_{n-1}$ commutes with all generators other than $r_{n-2}$.
Thus Proposition~\ref{prop:lifting} will apply to treat the induction step in the proof of Theorem~\ref{th:cximprimHam} for these cases.
It will be necessary to separately address the remaining infinite base cases,
\begin{align*}
G(d,1,2) & \hbox{ with } d \geq 2, \\
G(e,e,3) & \hbox{ with } e \geq 2, \\
G(de,e,2) & \hbox{ with } d,e \geq 2, \hbox{ and } \\
G(de,e,3) & \hbox{ with } d,e \geq 2. 
\end{align*}
 
In~\cite{BMR98} the case of $e=2$ in $G(de,e,n)$ is treated separately (though some indication was given it could be combined, we never resolved an issue concerning the double braid relation).
Here we treat separately the cases of $d=2$ in $G(de,e,2)$ and $G(de,e,3)$ but do not distinguish $e=2$.
Explicit Hamiltonian cycles are given for the families $G(d,1,2)$, $G(2e,e,2)$, and $G(e,e,3)$ in Lemmas~\ref{le:G(d,1,2)}, \ref{le:G(2e,e,2)}, and~\ref{le:G(e,e,3)}.
In Lemma~\ref{le:G(de,e,2)} we construct a Hamiltonian cycle in the graph for $G(de,e,2)$ by applying the flipping process to an explicit Hamiltonian path.
Lemmas~\ref{le:G(de,e,3)} and~\ref{le:G(2e,e,3)} achieve the lifting of the cycles given for $G(de,e,2)$ and $G(2e,e,2)$ to $G(de,e,3)$ and $G(2e,e,3)$ respectively.

Recall that we are using right Cayley graphs and all arithmetic is resolved with the appropriate modulus.
In the first two lemmas denote $r_1$ by $r$ and in all proofs denote the elements of the permutation group $S_3$ by:
\[\sigma_0=1, \quad \sigma_1=(1\,\,2), \quad \sigma_2=(2\,\,3), \quad \sigma_3=(1\,\,3), \quad \sigma_4=(1\,\,2\,\,3), \quad \sigma_5=(1\,\,3\,\,2).
\]

\begin{lemma}
\label{le:G(d,1,2)}
Let $G=G(d,1,2)$ with $d\geq 2$, which has presentation 
\[G=\langle t,r \mid t^d=r^2=1, trtr=rtrt \rangle.\] 
If $B=[[t]^{d-1},r]$, then $B^{2d}$
is a Hamiltonian cycle in $\Gamma(G,\{r,t\})$.
\end{lemma}
 
\begin{proof} 	
The elements of $G(d,1,2)$ are of the form 
\[\{(a,b \,\mid\, \sigma) \mid 0 \leq a,b < d \hbox{ and } \sigma \in S_2=\{\sigma_0,\sigma_1\}\}.\]
Entries in the tuples representing vertices are determined up to congruence modulo $d$. 
Denote the walk $B^{2d}\Sharp$ by $P$.
The number of vertices visited by $P$ is $2d^2$ which is the order of the group $G(d,1,2)$. 
It suffices to show that $B^{2d}$ is closed and $P$ is self-avoiding.

Let $v = t^{d-1}r$ and $w_i = t^i$ for $0 \leq i < d$. 
Since $t$ is of order $d$, the vertices in $w_0, \dots, w_{d-1}$ are distinct. 
In tuple notation, $t$ is $(1,0 \mid \sigma_0)$ and $r$ is $(0,0 \mid \sigma_1)$. 
From this, 
\[v = t^{d-1}r = (d-1,0 \mid \sigma_0)(0,0 \mid \sigma_1) = (-1,0 \mid \sigma_1),\] 
$v^2 = (-1,-1 \mid \sigma_0)$, and $v$ is of order $2d$, establishing
that $B^{2d}$ is closed.

Suppose $0 \leq m < 2d$, $0 \leq i,j < d$ and $v^m w_i = w_j$. Since $w_i = (i, 0 \mid \sigma_0)$ and $w_j = (j, 0 \mid \sigma_0)$, $m$ must be even, say $m = 2k$. 
Then $v^m w_i = w_j$ becomes $(-k + i,-k \mid \sigma_0) = (j,0 \mid \sigma_0)$. 
It follows that $k \equiv 0 \bmod{d}$ and $i \equiv j \bmod{d}$. 
Since $0 \leq i,j,k < d$ it must be that $k = 0$ and  $i = j$. 
Lemma 3.1 applies to allow us to conclude that $P = B^{2d} \Sharp$ is self-avoiding.
\end{proof}

\begin{lemma}
\label{le:G(2e,e,2)}
Let $G=G(2e,e,2)$ with $e\geq 3$, which has presentation
\[G=\langle r,s,t \mid r^2=s^2=t^2=1,\, tsr=srt,\, \underbrace{rtsrsrs\cdots}\limits_{e+1 \; \mathrm{factors}}=\underbrace{tsrsrs\cdots}\limits_{e+1 \; \mathrm{factors}} \rangle.\]  
If $B = [[[r,s]^{e-1}],r,t]$ then the walk $B^4$ is a Hamiltonian cycle in $\Gamma(G,\{r,s,t\})$.
\end{lemma}

\begin{proof} 	
The elements of $G(2e,e,2)$ are of the form 

\[\{(a,b \,\mid\, \sigma) \mid 0 \leq a,b < 2e, a + b \equiv 0 \bmod{e} \hbox{, and } \sigma \in S_2=\{\sigma_0,\sigma_1\}\}.\]

Entries in the tuples representing vertices are determined up to congruence modulo $2e$. 
We show that $B^4$ is closed and $P=B^4\Sharp$ is self-avoiding.

Let $w_{2i} = (rs)^i$ and $w_{2i+1} = (rs)^i r$ for $0 \leq i < e$. Let $v = (rs)^{e-1}rt = (e - 1, 1 \mid \sigma_1)$. 
It follows that $v^2 = (e, e \mid \sigma_0)$ so that $v$ is of order $4$ and $B^4$ is closed. 
The number of vertices visited by $P$ is $8e$ which is the order of the group $G(2e,e,2)$. 
It suffices to show that $P$ is self-avoiding, which we establish using Lemma \ref{lem:l2.6}.

Let $0 \leq m < 4$ and suppose we have $v^m w_i = w_j$. 
We may assume that $j$ is even, otherwise, right-multiply by $r$. 
It follows that the resulting $m$ and $i$ must have the same parity. 
There are two cases to consider:

Case 1:  $m$ is even. Set $m = 2l$ for $l = 0 \text{ or } 1$ and $i = 2 l_i$, $j = 2 l_j$. 
Then $v^m w_i = w_j$ becomes 
\[(l e + l_i, l e - l_i \mid \sigma_0) = (l_j, -l_j \mid \sigma_0).\]

This yields a system of two congruences that imply $l_i - l_j \equiv 0 \bmod{e}$, forcing $l_i = l_j$ and $i = j$.
It follows that $l$ must be $0$ and so $m = 0$.

Case 2:  $m$ is odd. Set $m = 2 l+1$ for $l = 0 \text{ or } 1$ and $i = 2l_i+1$, $j = 2 l_j+1$. 
Then $v^m w_i = w_j$ becomes 
\[((l+1)e-l_i-1, le+l_i+1 \mid \sigma_0) = (l_j, -l_j \mid \sigma_0) \]
yielding the system of congruences
\begin{align*}
    e+le - l_i - 1 &\equiv \phantom{-}l_j \bmod{2e}\\
    l e + l_i + 1 &\equiv -l_j \bmod{2e}
  \end{align*}
which has no solution.
Lemma 3.1 applies to allow us to conclude that $P$ is self-avoiding.
 \end{proof}
 
In the next lemma, to avoid subscripts, continue to denote $r_1$ by $r$ and denote $r_2$ by $q$. 
 
\begin{lemma}
\label{le:G(e,e,3)}
Let $G=G(e,e,3)$ with $e\geq 2$, which has presentation 
\[G=\langle s,r,q \mid s^2=r^2=q^2=1,\, sqs=qsq,\, rqr=qrq, qsrqsr=srqsrq,\, \underbrace{srs\cdots}\limits_{e \; \mathrm{factors}}=\underbrace{rsr\cdots}\limits_{e \; \mathrm{factors}}. \rangle\] 
Let $A = [q,s,q,r,q,r]$ and $B=[A^e \Sharp, s]$.  
Then $B^e$ is a Hamiltonian cycle in $\Gamma = \Gamma(G,\{s,r,q\})$.
\end{lemma}
\begin{proof}
The elements of $G(e,e,3)$ are of the form 
\[\{(a,b,c \,\mid\, \sigma) \mid 0 \leq a,b,c < e, a + b + c \equiv 0 \bmod{e} \hbox{, and } \sigma \in S_3=\{\sigma_0,\sigma_1,\dots,\sigma_5\}\}.\]
Entries in the tuples are determined up to congruence modulo $e$. 

Let $1=w_0, w_1, \dots, w_{k-1}$ denote the vertices, in sequence, of the walk $A^e \Sharp$ and let $v = w_{k-1}s$. Note that $k = 6e$. In tuple notation, 
\begin{alignat*}{2}
q &= (0,0,0 \mid \sigma_2), &\qquad qsqr &= (-1,0,1 \mid \sigma_5), \\
qs &= (-1,0,1 \mid \sigma_4), &\qquad qsqrq &= (-1,0,1 \mid \sigma_1), \\ 
qsq &= (-1,0,1 \mid \sigma_3), &\qquad qsqrqr &= (-1,0,1 \mid \sigma_0).
\end{alignat*}
Note that all $\sigma \in S_3$ appear in this list. 

The walk determined by $A^e\Sharp s$ ends in the vertex $v = (qsqrqr)^e rs = (1, -1, 0 \mid \sigma_0)$ and $v$ has order $e$. 
Suppose $v^m w_i = w_j$. Write $w_i = (qsqrqr)^{l_i} v_i$ and $w_j = (qsqrqr)^{l_j} v_j$ where $v_i, v_j \in \{1, q, qs, qsq, qsqr, qsqrq\}$. 
It follows that the tuples for $v_i$ and $v_j$ must contain the same $\sigma$ and consequently must be equal. Cancellation results in a reduction to either 
\[v^m (qsqrqr)^l = 1 \hbox{ or } v^m = (qsqrqr)^l\]
where $l = \vert l_i - l_j \vert $.

The first equation asserts that $(m - l,-m,l \mid \sigma_0) = (0,0,0 \mid \sigma_0)$ forcing $m = l = 0$ and $w_i = w_j$. The second equation asserts
that $(m, -m, 0 \mid \sigma_0) = (-l, 0, l \mid \sigma_0)$, again forcing $m = l = 0$ and $w_i = w_j$. Lemma 3.1 allows us to conclude that $P$ is self-avoiding.
The order of the group $G$ is $6e^2$ and $v$ is of order $e$ so $B^e$ is closed and is a Hamiltonian cycle.
\end{proof}

\begin{lemma}
\label{le:G(de,e,2)}
Let $G=G(de,e,2)$ with $d \geq 3, e\geq 2$, which has presentation
\[G=\langle t,s,r \mid r^2=s^2=t^d=1,\, tsr=srt,\, \underbrace{rtsrsrs\cdots}\limits_{e+1 \; \mathrm{factors}}=\underbrace{tsrsrs\cdots}\limits_{e+1 \; \mathrm{factors}} \rangle.\]  
Let $A = [[t]^{d-1},s]^{2d}$, $B=[A\Sharp, r]$ and $S = \{t, s, r\}$. 
Then  $B^e\Sharp$ is a Hamiltonian path in $\Gamma(G,S)$ and two successive flips of this path with respect to $s$ produce a Hamiltonian path $P$ such that $[P, r]$ is a Hamiltonian cycle in $\Gamma(G,S)$.
\end{lemma}

\begin{proof}

Let $H = \langle s, t \rangle$. 
An elementary argument establishes that the vertex set of the walk $A$ is $H$. 
The proof of Lemma 4.1 can be modified slightly to provide an argument showing that 
$A $ is a Hamiltonian cycle in $\Gamma(H,\{s,t\})$. It follows that $|H|=2d^2$ and $H$ has index $e$ in $G$.

The right cosets of $H$ are $H, Hsr, H(sr)^2, \dots , H(sr)^{e-1}$. To see this, suppose for some $0 \leq i, j < e$, that 
$(sr)^i$ and $(sr)^j$ determine the same right coset of H. Assume $i \geq j$. It follows
that $(sr)^{k} \in H$ where $k = i-j$. Elements of $H$ of the form $(a, b \mid \sigma_0)$ must arise as $(t^{d-1}s)^{2m}t^n$ for suitable $m, n$.
But $(sr)^k = (-k, k \mid \sigma_0)$ while $(t^{d-1}s)^{2m}t^n = (-me + ne, -me \mid \sigma_0)$.
If $(sr)^k=(t^{d-1}s)^{2m}t^n$ then $k\equiv -me \bmod{de}$ and hence  
$k$ must be congruent to 0 modulo $e$, but since $0 \leq k < e$, it must be that $k = 0$. 

The graph $\Gamma(G, S)$ is partitioned into $e$ subgraphs corresponding to the cosets of $H$. 
Each of these subgraphs has a spanning cycle that is obtained by applying $A$ starting at each of $1, sr, (sr)^2, \dots (sr)^{e-1}$. 
If the final $s$ in each of these cycles is replaced by $r$ the effect 
is the same as the concatenation of $A$ with $[s,r]$, effectively moving to the next coset. 
Since $s$ is of order 2, $[A,s,r] = [A \Sharp ,r] = B$. 
Thus, $B^e \Sharp$ utilizes each subgraph cycle to form a Hamiltonian path in $\Gamma(G, S)$. 

The Hamiltonian path $B^e \Sharp$ ends at the vertex $(-e,e \mid \sigma_1)$. It remains to show how to alter this path so that it ends in a vertex adjacent to $1$.
The cosets of $H$ are closed with respect to right multiplication by $s$ so that 
flipping the path $B^e \Sharp$ with respect to $s$ only alters the segment 
spanning the last coset of $H$. The first flip using $s$ reverses the cycle (path) in the last coset, and ends
at 
\[(-(e-1),e-1 \mid \sigma_0) t = (1, e-1 \mid \sigma_0).\] 
This reversal changes each $t$ to $t^{-1}$ in the portion of the path that is traversed backwards. 
The second  flip using $s$ produces a path that ends at
\[(1, e-1 \mid \sigma_0) s t^{-1} = (0, e \mid \sigma_1)t^{-1} = (0, 0 \mid \sigma_1),\] 
so the terminal vertex of the path is now adjacent, via $r$, to $1$. Refer to Figure~\ref{fig:flipping}. 
\end{proof}
\begin{figure} [h!]
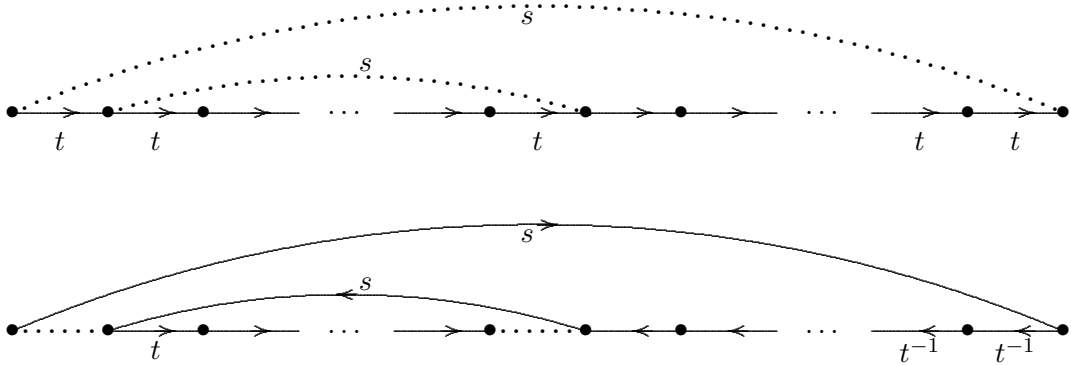

\begin{center}
\mbox
{
{
\beginpicture
     \setcoordinatesystem units <0.5in,0.5in>
  
    \setsolid
    \plot -5 1 -2 1 /
    \plot -1 1 3 1 /
    \plot  4 1 6 1 /

    \arrow <6pt> [.2,.67] from -4.5 1 to -4.3 1     
    \arrow <6pt> [.2,.67] from -3.5 1 to -3.3 1     
    \arrow <6pt> [.2,.67] from -2.5 1 to -2.3 1     
    \arrow <6pt> [.2,.67] from -.5 1 to -.3 1     
    \arrow <6pt> [.2,.67] from .5 1 to .7 1     
    \arrow <6pt> [.2,.67] from 1.5 1 to 1.7 1     
    \arrow <6pt> [.2,.67] from 2.5 1 to 2.7 1     
    \arrow <6pt> [.2,.67] from 4.5 1 to 4.7 1     
    \arrow <6pt> [.2,.67] from 5.5 1 to 5.7 1       
    
    \put{$t$} at -4.5 .7
    \put{$t$} at -3.5 .7
    
    \put{$t$} at .5 .7

    \put{$t$} at 4.5 .7
    \put{$t$} at 5.5 .7

    \put{$s$} at .39 2
    \put{$s$} at -1.3 1.5

    \put {$\bullet$} at -5 1
    \put {$\bullet$} at -4 1
    \put {$\bullet$} at -3 1
    
    \put {$\dots$} at -1.5 1
    
    \put {$\bullet$} at 0 1
    \put {$\bullet$} at 1 1
    \put {$\bullet$} at 2 1

    \put {$\dots$} at 3.5 1

    \put {$\bullet$} at 5 1
    \put {$\bullet$} at 6 1
      
     \setdots
     \setplotsymbol({\Large .})
     \circulararc -46 degrees from -5 1 center at .5 -12
     \setplotsymbol({\tiny -})
     \setsolid
     
     \setdots
     \setplotsymbol({\Large .})
     \circulararc -35 degrees from -4 1 center at -1.5 -7
     \setplotsymbol({\tiny -})
     \setsolid
       
     \endpicture
}
}

\vskip .4in
\mbox
{
{
\beginpicture
     \setcoordinatesystem units <0.5in,0.5in>
  
    \setsolid
    \plot -4 1 -2 1 /
    \plot -1 1 0 1 /
    \plot  1 1 3 1 /
    \plot  4 1 6 1 /

    \put{$t$} at -3.5 .8
    

    \put{$t^{-1}$} at 4.5 .8
    \put{$t^{-1}$} at 5.5 .8

    \put{$s$} at .39 2
    \put{$s$} at -1.3 1.5

    \arrow <6pt> [.2,.67] from -3.5 1 to -3.3 1     
    \arrow <6pt> [.2,.67] from -2.5 1 to -2.3 1     
    \arrow <6pt> [.2,.67] from -.5 1 to -.3 1     
    \arrow <6pt> [.2,.67] from 1.7 1 to 1.5 1     
    \arrow <6pt> [.2,.67] from 2.7 1 to 2.5 1     
    \arrow <6pt> [.2,.67] from 4.7 1 to 4.5 1     
    \arrow <6pt> [.2,.67] from 5.7 1 to 5.5 1     

    \arrow <6pt> [.2,.67] from .5 2.11 to .7 2.11    
    \arrow <6pt> [.2,.67] from -1.4 1.38 to -1.6 1.38

    \put {$\bullet$} at -5 1
    \put {$\bullet$} at -4 1
    \put {$\bullet$} at -3 1
    
    \put {$\dots$} at -1.5 1
    
    \put {$\bullet$} at 0 1
    \put {$\bullet$} at 1 1
    \put {$\bullet$} at 2 1

    \put {$\dots$} at 3.5 1

    \put {$\bullet$} at 5 1
    \put {$\bullet$} at 6 1

     \circulararc -46 degrees from -5 1 center at .5 -12
     
     \circulararc -35 degrees from -4 1 center at -1.5 -7
     
     \setdots
     \setplotsymbol({\Large .})

     \plot  -5 1 -4 1 /
     \plot  0 1 1 1 /
     
     \setplotsymbol({\tiny})
     \setsolid

     \endpicture
}
}
\caption {\small Altering the Hamiltonian path in Lemma \ref{le:G(de,e,2)}}
\label{fig:flipping}
\end{center}
\end{figure}

\begin{lemma}
\label{le:G(de,e,3)}
Let $G=G(de,e,3)$ with $d\geq 3, e\geq 2$, which has presentation
\begin{align*}
G=\langle t,s,r_1,r_2 \mid r_1^2=r_2^2=s^2=t^d&=1,\, tr_2=r_2t,\, tsr_1=sr_1t
,\, sr_2s=r_2sr_2,\, r_1r_2r_1=r_2r_1r_2, \\
\underbrace{r_1tsr_1sr_1s\cdots}\limits_{e+1 \; \mathrm{factors}}&=\underbrace{tsr_1sr_1s\cdots}\limits_{e+1 \; \mathrm{factors}},\,\, r_2sr_1r_2sr_1=sr_1r_2sr_1r_2 \rangle.
\end{align*}
Then $\Gamma(G,\{t,s,r_1,r_2\})$ has a Hamiltonian cycle.
\end{lemma}

\begin{proof}

Let $H = \langle t,s,r_1 \rangle$ and observe that $H$ is isomorphic to
$G(de,e,2)$. Set $S = \{t,s,r_1,r_2\}$. The components, $\{\Gamma_l\}$, of 
$\Gamma(G,S \setminus \{r_2\})$ are subgraphs on the various cosets of $H$ with $r_2$-labeled edges removed.
Since $r_2$ is of order two, we have a setting consistent with Corollary \ref {cor:badness}. 

For the cycle $C_H$ we take the Hamiltonian cycle constructed in Lemma \ref{le:G(de,e,2)}. 
Using the notation of Lemma \ref {le:G(de,e,2)}, $r_1$ is matched with $r$. 
Since $t$ and $r_2$ commute, the value of $bad(C_H,r_2)$ is found by counting the number of occurrences of consecutive $rs$ or $sr$ edge pairs. 
No such pairs appeared in the initial Hamiltonian path used in the construction of $C_H$.
At most two such pairs are introduced by the two $s$-flips and they must occur at the beginning and end of the altered path through the final coset (examine Figure~\ref{fig:flipping}). 
Consequently, $bad(C_H,r_2) \leq 2$. 

Suppose $\Gamma_i$ and $\Gamma_j$ are connected by an $r_2$ edge. Let $H_i = g_iH$ and $H_j = g_jH$ 
denote the left cosets that are the vertex sets of $\Gamma_i$ and $\Gamma_j$, respectively. Then
there are $h_i, h_j \in H$ with $g_ih_ir_2 = g_jh_j$. The generators $t$ and $r_2$ commute so that for $0 \leq k < d$ we have
\[ g_ih_it^kr_2 = g_ih_ir_2t^k = g_jh_jt^k\]
and each $(g_ih_it^k,g_jh_jt^k)$ is an edge linking $\Gamma_i$ and $\Gamma_j$. This shows that $c(i,j) \geq d$. Since $d \geq 3$ and $bad(C_H,r_2) \leq 2$,
Corollary \ref{cor:badness} applies and we conclude that $\Gamma(G,\{t,s,r_1,r_2\})$ has a Hamiltonian cycle.
\end{proof}

\begin{lemma}
\label{le:G(2e,e,3)}
Let $G=G(2e,e,3)$ with $e\geq 2$, which has presentation
\begin{align*}
G=\langle t, s, r_1, r_2 \mid r_1^2=r_2^2=s^2=t^2&=1,\, tr_2=r_2t,\, tsr_1=sr_1t
,\, sr_2s=r_2sr_2,\, r_1r_2r_1=r_2r_1r_2, \\
\underbrace{r_1tsr_1sr_1s\cdots}\limits_{e+1 \; \mathrm{factors}}&=\underbrace{tsr_1sr_1s\cdots}\limits_{e+1 \; \mathrm{factors}},\,\, r_2sr_1r_2sr_1=sr_1r_2sr_1r_2 \rangle.
\end{align*}
Then $\Gamma(G,\{t,s,r_1,r_2\})$ has a Hamiltonian cycle.
\end{lemma}

\begin{proof}
Let $S = \{t, s, r_1, r_2\}$ and $H = \langle S \setminus \{r_2\} \rangle$. Observe that each generator in $S$ is of order 2 and that 
$H$ is isomorphic to $G(2e,e,2)$. 
Let $\{\Gamma_l\}$ denote the components of $\Gamma(G,S \setminus \{r_2\})$. Each $\Gamma_ l $ has as vertex set some left coset of $H$ and
all are isomorphic to $\Gamma_H = \Gamma(H,S \setminus \{r_2\})$. 
Let $C_H$ denote the Hamiltonian cycle in $\Gamma_H$ described in Lemma \ref{le:G(2e,e,2)}. 
Since $bad(C_H,r_2)$ is quite large (its value is $8(e-1)$) the procedure used in the proof
of Lemma \ref{le:G(de,e,3)} will fail. Our approach here involves chaining together self-avoiding paths derived from $C_H$ into a large cycle $C$ which spans $4e$ of the 
$6e$ coset graphs. This large cycle can  
then be combined with spanning cycles in the remaining components using the commutative joining property to obtain the desired result.

The left cosets of $H$ are described as 
\begin{align*}
 H(k,1) &= \{(k,b,c\mid\sigma) \in G, \sigma \in \{\sigma_3,\sigma_5\} \},\\
 H(k,2) &= \{(a,k,c\mid\sigma) \in G, \sigma \in \{\sigma_2,\sigma_4\} \},\hbox{ and }\\
 H(k,3) &= \{(a,b,k\mid\sigma) \in G, \sigma \in \{\sigma_0,\sigma_1\} \},
\end{align*} 
where $0 \leq k < 2e$. For example, $H = H(0,3)$.

From Lemma \ref{le:G(2e,e,2)}, $C_H = [[r_1,s]^{e-1}],r_1,t]^4$. It is easy  to verify that $(r_1s)^{e-1}r_1t = ts$.  It follows that the 
path $Q$ obtained from $C_H \Sharp$ 
by flipping with respect to $s$ is given by
\[Q = [[r_1,s]^{e-1},r_1,t,s,[[r_1,s]^{e-1},r_1,t]^2,[r_1,s]^{e-1}].\]
If  $Q$ is applied starting at 1 the ending vertex is $tsr_1$. See Figure~\ref{fig:flip3}.

\begin{figure} [b!]
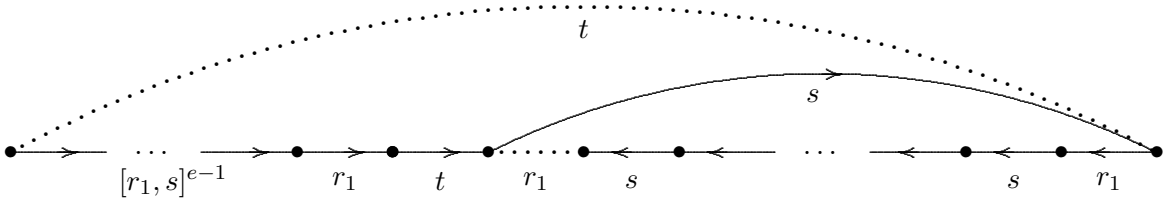

\begin{center}
\mbox
{
{
\beginpicture
     \setcoordinatesystem units <0.5in,0.5in>
  
    \setsolid
    \plot -5 1 -4 1 /
    \plot -3 1 0 1 /

    \plot 1 1 3 1 /
    \plot  4 1 7 1 /

    \arrow <6pt> [.2,.67] from -4.5 1 to -4.3 1     
    \arrow <6pt> [.2,.67] from -2.5 1 to -2.3 1     
    \arrow <6pt> [.2,.67] from -1.5 1 to -1.3 1     
    \arrow <6pt> [.2,.67] from -.5 1 to -.3 1     
    \arrow <6pt> [.2,.67] from 1.5 1 to 1.3 1     
    \arrow <6pt> [.2,.67] from 2.5 1 to 2.3 1     
    \arrow <6pt> [.2,.67] from 4.5 1 to 4.3 1     
    \arrow <6pt> [.2,.67] from 5.5 1 to 5.3 1     
    \arrow <6pt> [.2,.67] from 6.5 1 to 6.3 1

    \arrow <6pt> [.2,.67] from 3.5 1.82 to 3.7 1.82     
    \put{$s$} at 3.4 1.6
    \put{$t$} at 1 2.3
        
    \put{$r_1$} at -1.5 .7
    \put{$t$} at -.5 .7
    \put{$r_1$} at .5 .7
    \put{$s$} at 1.5 .7
    
    \put{$r_1$} at 6.5 .7
    \put{$s$} at 5.5 .7

    

    \put{$[r_1,s]^{e-1}$} at -3.3 .7

    \put {$\bullet$} at -5 1
    \put {$\bullet$} at -1 1
    \put {$\dots$} at -3.5 1
    \put {$\bullet$} at -2 1
    \put {$\bullet$} at -1 1
    
    
    \put {$\bullet$} at -1 1
    \put {$\bullet$} at 0 1
    \put {$\bullet$} at 1 1
    \put {$\bullet$} at 2 1

    \put {$\dots$} at 3.5 1

    \put {$\bullet$} at 5 1
    \put {$\bullet$} at 6 1
    \put {$\bullet$} at 7 1
     
     \circulararc 53 degrees from 7 1 center at 3.5 -6
     
     \setdots
     \setplotsymbol({\Large .})
     \circulararc -57 degrees from -5 1 center at 1 -10

     \plot 0 1 1 1 /

     \setplotsymbol({\tiny -})
     \setsolid
     

     \endpicture
}
}

\caption {\small Adapting the Hamiltonian cycle from Lemma \ref{le:G(2e,e,2)}}
\label{fig:flip3}
\end{center}
\end{figure} 

From $t s r_1 r_2=(e-1,1,0\mid\sigma_2)$ we see that
$$\begin{aligned}
(t s r_1 r_2)^{2i} &= (i(2e-2),i,i\mid\sigma_0) \in H(i,3) \text{ and}\\
(t s r_1 r_2)^{2i+1} &= (i(2e-2)+e-1,i+1,i \mid\sigma_2) \in H(i+1,2).   
  \end{aligned}
$$
It follows that 
$tsr_1r_2$ has order $4e$ and the cycle $[t,s,r_1,r_2]^{4e}$ visits each coset of the form $H(k,2)$ or $H(k,3)$.
If each factor of $[t, s, r_1]$ is replaced by $Q$ the result is a cycle $C$ that spans $4e$ of the $6e$ different coset graphs.
We use the fact that $t$ and $r_2$ commute and complete the argument by showing that any component $\Gamma_l$ disjoint from $C$ contains a spanning cycle $C_l$
such that $C$ and $C_l$ have the commutative joining property with respect to $r_2$. 

We first
establish that for any $k$ with $0 \leq k < 2e$ there is an edge $(v,vt)$ in $C$ with $vr_2 \in H(k,1)$.
To accomplish this we observe that
$$(a,b,c\mid\sigma_1) r_2 = (a,b,c\mid \sigma_5) \hbox{ and } (a,b,c\mid \sigma_4) r_2 = (a,b,c\mid \sigma_3)$$ 
and so it suffices to show that for any $k$ with $0 \leq k < 2e$ there
is an edge $(v,v t)$ in $C$  with $v$ of the form $(k,b,c\mid\sigma_1)$ or $(k,b,c\mid\sigma_4)$.

The path $Q$ contains exactly three edges of the form $(v,vt)$ and the vertex $v$ is one of
\[
(r_1s)^{e-1}r_1,\quad (r_1s)^{e-1}r_1ts(r_1s)^{e-1}r_1 \qquad \text{ or }\quad (r_1s)^{e-1}r_1ts(r_1s)^{e-1}r_1t(r_1s)^{e-1}r_1.
\]
Only the first two of these vertices are adjacent via $r_2$ to a vertex in some $H(k,1)$. 
Using $(r_1s)^{e-1}r_1t = ts$ and the fact that all generators are of order two, the two useful vertices are $tst$ and $st$.
The complete collection of edges $(v,vt)$ in $C$ where $v r_2$ is not a vertex in $C$ corresponds to the set of vertices of one of the forms $(tsr_1r_2)^i tst$ or $(tsr_1r_2)^i st$ for $ 0  \leq i < 4e$. 
It is convenient to split into the cases where $i$ is odd and where $i$ is even. 
For $0 \leq j < 2e$, direct computations produce:
 \begin{align*}
      (tsr_1r_2)^{2j} tst &= (e-2j-1,e+j+1,j \mid \sigma_1),\\
      (tsr_1r_2)^{2j+1} tst &= (-2j-2,j+1,e+j+1 \mid \sigma_4),\\
      (tsr_1r_2)^{2j} st &= (-2j-1,e+j+1,j \mid \sigma_1), \hbox{ and } \\
      (tsr_1r_2)^{2j+1} st &= (e-2j-2,j+1,e+j+1 \mid \sigma_4).     
 \end{align*}

The set $\{-2j-2,-2j-1
\mid 0 \leq j < 2e\}$ will contain a complete set of residues modulo $2e$ and, consequently, for any $k$ there is an edge $(v,vt)$ in $C$ where $vr_2 \in H(k,1)$. 

We have shown that if $\Gamma_l$ is a coset graph disjoint from $C$ then there is an edge $(v,vt)$ in $C$ and $vr_2$ is a vertex in $\Gamma_l$. 
The vertex set of $\Gamma_l$ is a left coset of $H$ so that the cycle $(v r_2)C_H$ is a spanning cycle in $\Gamma_l$. 
Vertex transitivity can be used to alter this cycle to obtain a spanning cycle $C_l$ in $\Gamma_l$ containing the edge $(vr_2,vr_2t)$. 
It follows that $C$ and $C_l$ have the 
commutative joining property with respect to $r_2$ and Proposition \ref{prop:multi-join} applies to complete the proof. 
\end{proof}

We have thus established Hamiltonicity in each of the required base cases. The remaining graphs under consideration can be handled by induction to establish our main result.
\begin{theorem}
\label{th:cximprimHam}
If $G=G(de,e,n)$ is an irreducible imprimitive complex reflection group and $S$ is a standard generating set for $G$ (as given in Section \ref{sec:background}), then the (undirected right) Cayley graph $\Gamma(G,S)$ has a Hamiltonian cycle.
\end{theorem}

\begin{proof}
Since $G=G(de,e,n)$ is an irreducible imprimitive complex reflection group, then $de\geq 2$, $n\geq 1$, and $(de,e,n)\neq (2,2,2)$.
When $n=1$ the group is cyclic so $\Gamma(G,S)$ is circulant and has a Hamiltonian cycle.
Lemmas~\ref{le:G(d,1,2)}-\ref{le:G(2e,e,3)} prove the statement for $n=2$ and $n=3$. 
Proceed by induction on $n$.
The generating set for $G = G(de,e,n)$ being considered can be written as $S \cup \{r_n\}$ where the subgroup $H$ of $G$ generated by $S$ is isomorphic to $G(de,e,n-1)$, $r_n$ commutes with all but one of the elements of $S$ and that element is of order two. 
By induction, there is a Hamiltonian cycle $C_H$ in 
$\Gamma(H, S)$. 
This cycle can contain no consecutive edge labels both of which fail to commute with $r_n$, i.e., $bad(C_H,r_n) = 0$.  
Thus by Proposition~\ref{prop:lifting}, $\Gamma_G$ also contains a Hamiltonian cycle.
\end{proof}

\section*{Acknowledgements}
We thank Dave Morris for helpful conversations.

\section*{References}

\bibliographystyle{alpha}
\bibliography{HamCycleDMEdVersion}

\end{document}